\documentclass{amsart}

\usepackage{amscd}
\usepackage{amsmath}
\usepackage{amssymb}
\usepackage{amsthm}
\usepackage{epsf}
\usepackage{latexsym}
\usepackage{verbatim}
\usepackage[all, cmtip]{xy}
\usepackage{tikz}
\usepackage{hyperref}
\usetikzlibrary{positioning}
\usetikzlibrary{matrix}

\tikzstyle{bsq}=[rectangle, draw, thick, minimum width=1cm, minimum height=1cm]
\tikzstyle{bver}=[rectangle, draw, thick, minimum width=1cm, minimum height=2cm]
\tikzstyle{bhor}=[rectangle, draw, thick, minimum width=2cm, minimum height=1cm]

\newtheorem{theorem}{Theorem}[section]
\newtheorem{lemma}[theorem]{Lemma}

\newtheorem{corollary}[theorem]{Corollary}
\newtheorem{proposition}[theorem]{Proposition}

\newtheorem{varexample}[theorem]{Example}

\theoremstyle{definition}

\newtheorem{definition}[theorem]{Definition}

\newcommand{\ZZ}{\mathbb{Z}}

\newcommand{\trop}{\operatorname{trop}}

\newcommand{\val}{\operatorname{val}}

\newcommand{\outdeg}{\mathrm{outdeg}}

\newcommand{\gon}{\operatorname{gon}}
\newcommand{\supp}{\operatorname{supp}}

\title{Gonality of Expander Graphs}
\author{Neelav Dutta}
\author{David Jensen}
\begin{document}
\maketitle

\begin{abstract}
We provide lower bounds on the gonality of a graph in terms of its spectral and edge expansion.  As a consequence, we see that the gonality of a random 3-regular graph is asymptotically almost surely greater than one seventh its genus.
\end{abstract}

\section{Introduction}

In this paper, we study the relationship between the gonality of a graph $G$ and its expansion properties.  In particular, we provide lower bounds on the gonality $\gon (G)$ in terms of the Cheeger constant $h(G)$.  We refer the reader to Section \ref{Sec:Prelim} for definitions of these graph invariants.  Our main result is the following.

\begin{theorem}
\label{Thm:LowerBound}
For any $u \in (0, \frac{1}{2}]$, let $B_{u}(G)$ be the smallest degree of an effective divisor $D$ such that every connected component of $V(G) \smallsetminus \supp(D)$ has size at most $u \vert V(G) \vert $. Then
\[
gon(G)\geq \min \left\{ B_{u}(G), h(G)u \vert V(G) \vert \right\}.
\]
\end{theorem}

Note that $B_u (G)$ is decreasing in $u$, while $h(G)u \vert V(G) \vert$ is increasing in $u$.  As a consequence of Theorem \ref{Thm:LowerBound}, we obtain the following lower bound on the gonality of a random 3-regular graph.

\begin{theorem}
\label{Thm:RandomThreeRegular}
Let $G$ be a random 3-regular graph on $n$ vertices.  Then
\[
\gon(G) \geq 0.072n
\]
asymptotically almost surely.
\end{theorem}

Note that the genus of a 3-regular graph on $n$ vertices is $\frac{1}{2}n+1$.  Theorem \ref{Thm:RandomThreeRegular} therefore implies that the expected gonality of a random 3-regular graph is at least one seventh its genus.  This result is obtained by bounding the invariants $h(G)$ and $B_u (G)$ that appear in Theorem \ref{Thm:LowerBound} for random 3-regular graphs.  This yields a lower bound on the gonality of such graphs that depends on $u$, which is optimized for some $u$ between $0.35$ and $0.40$.

Our primary motivation for studying the gonality of graphs comes from the theory of algebraic curves.  In \cite{Baker08}, Baker develops a theory of specialization of divisors from curves to graphs.  This theory allows one to establish results in algebraic geometry using combinatorial techniques, and conversely, to discover new combinatorics using algebraic geometry.  Because the top-dimensional strata of $M_g^{\trop}$ correspond to 3-regular graphs, such graphs are considered to be the closest analogues of algebraic curves.  It is therefore natural to compare invariants of random 3-regular graphs to the analogous invariants of general curves.

By \cite{Kempf71, KleimanLaksov72}, the gonality of any curve of genus $g$ is at most $\lfloor \frac{g+3}{2} \rfloor$.  Moreover, the Brill-Noether theorem implies that equality holds for the general curve \cite{GriffithsHarris80}.  In other words, the set of genus $g$ curves of gonality exactly $\lfloor \frac{g+3}{2} \rfloor$ is a dense open subset of the moduli of space of curves $M_g$.  One could hope to address the analogous questions for 3-regular graphs.  First, what is the maximum gonality of a 3-regular graph?  Second, what is the expected gonality of a random 3-regular graph?  Theorem \ref{Thm:RandomThreeRegular} represents progress toward the second of these questions.

Theorem \ref{Thm:RandomThreeRegular} is an improvement on several earlier results.  In \cite{treewidth}, de Bruyn and Gijswijt show that the gonality of a graph is bounded below by its treewidth.  In \cite{KM93}, the authors show that a random 3-regular graph on $n$ vertices almost surely has bisection width at least $\frac{10}{99}n$.  As noted in \cite{NM08}, this implies that such a graph almost surely has treewidth at least $\frac{5}{99}n$.  In \cite{AminiKool}, Amini and Kool use the results of \cite{treewidth} to produce a lower bound on graph gonality in terms of the algebraic connectivity.  They then use this to show that the gonality of a random $k$-regular graph is bounded above and below by linear functions in $n$.  While \cite{AminiKool} do not attempt to bound the implied constants, their approach cannot yield a bound higher than $\frac{5}{99}n$, due to their reliance on \cite{treewidth}.  In a related direction, \cite{AminiKool} shows that the expected gonality of an Erdos-Renyi random graph on $n$ vertices is also bounded above and below by functions that are linear in $n$.  An improvement of this result appears in \cite{randomgonality}, where it is shown that the expected gonality of such graphs is in fact asymptotic to $n$.

The argument of the present paper is in some ways similar to the proof in \cite{treewidth} that gonality is bounded below by treewidth.  The treewidth of a graph can be defined in terms of certain collections of connected subsets of the vertices, known as brambles.  One example of a bramble is the set of all connected subsets of $V(G)$ of size at least $\frac{1}{2} \vert V(G) \vert$, which implies that the treewidth of a graph $G$, and hence its gonality, is at least $B_{\frac{1}{2}} (G)$.  To improve on this, we note that \cite{treewidth} in fact uses the treewidth in two separate ways.  First, the treewidth is a lower bound on the degree of an effective divisor that intersects every element of a bramble.  Second, it provides a lower bound on the size of a cut that separates two elements of this bramble.  If we consider instead the set of all connected subsets of size at least $u \vert V(G) \vert$ for some $u \leq \frac{1}{2}$ (which is not a bramble), then the minimum size of a cut that separates these subsets decreases, while the minimum size of an effective divisor that intersects every subset increases.  By varying $u$, it is possible to obtain a higher bound.

The invariant $B_u (G)$ is somewhat mysterious, and much of the paper is devoted to finding bounds for it in terms of more familiar graph invariants, including the $u$-Cheeger constant $h_u (G)$ and the algebraic connectivity $\lambda_2$.  We again refer the reader to Secion \ref{Sec:Prelim} for definitions of these graph invariants.  Specifically, we establish the following results.

\begin{theorem}
\label{Thm:CheegerBound}
Let $G$ be a $k$-regular graph. Then, for any $u$, we have
\[
\gon(G) \geq \min \left\{ \frac{h_{u}(G)}{k+h_{u}(G)} \vert V(G) \vert , h(G)u \vert V(G) \vert \right\} .
\]
\end{theorem}

\begin{theorem}
\label{Thm:SpectralBound}
Let $G$ be a graph, and let $d$ be the maximum valence of a vertex in $G$. Then
\[
\gon(G) \geq \frac{\vert V(G) \vert}{2\lambda_2} \left[ -(7\lambda_2 + 9d) + 3\sqrt{9\lambda_2^2 + 14d\lambda_2 + 9d^2} \right] .
\]
\end{theorem}

In the case of primary interest to us, namely that of random 3-regular graphs, the bound provided by Theorem \ref{Thm:CheegerBound} is stronger than that of Theorem \ref{Thm:SpectralBound}.  Nevertheless, Theorem \ref{Thm:SpectralBound} has at least two advantages.  First, while Theorem \ref{Thm:CheegerBound} applies only to regular graphs, Theorem \ref{Thm:SpectralBound} applies to arbitrary graphs.  Second, while computation of the $u$-Cheeger constant is NP-hard, the algebraic connectivity can be computed, to any degree of accuracy, in polynomial time.  For this reason we expect that Theorem \ref{Thm:SpectralBound} may be more useful for applications, as it is more efficient.

While this paper is primarily concerned with lower bounds on graph gonality, we make a few brief remarks on upper bounds.  It follows from Baker's specialization lemma that, if $G$ is a graph of genus $g$, then there exists a positive integer $e$ such that the refinement obtained by subdividing each edge of $G$ into $e$ edges has gonality at most $\lfloor \frac{g+3}{2} \rfloor$.  We note that this argument uses the Kleiman-Laksov result on algebraic curves and to date there is no known purely combinatorial proof.  The question of whether this statement holds without refinement -- that is, whether the integer $e$ can always be taken to be 1 -- remains open.  Without passing to refinements, much less is known.  As far as we are aware, the best known upper bound for the gonality of a graph is its genus.  Specifically, if $G$ is a graph of genus $g$ and $E$ is any effective divisor of degree $g-2$, then by Riemann-Roch the divisor $K_G-D$ has rank at least 1.  It follows that $\gon (G) \leq g$.  To obtain another upper bound, note that the complement of an independent set has rank at least 1, and thus $\gon (G) \leq \vert V(G) \vert - \alpha (G)$, where $\alpha (G)$ is the size of the largest independent set in $G$.  For a 3-regular graph $G$, however, this bound is smaller than $g$ if and only if $G$ is bipartite.  In this case $\vert V(G) \vert - \alpha (G) = g-1$, so for bipartite 3-regular graphs this upper bound differs by only 1 from the bound obtained by Riemann-Roch.

\subsection*{Acknowledgements}  The first author was supported by a gift from the Eaves family for summer undergraduate research.  We thank the Eaves for their generosity.  The second author was supported in part by NSF DMS-1601896.

\section{Preliminaries}
\label{Sec:Prelim}

\subsection{Divisor Theory of Graphs}
In this section, we outline the basic definitions and properties of divisors on graphs.  For more details, we refer the reader to \cite{Divisors, Baker08}.

\begin{definition}
A \emph{divisor} $D$ on a graph $G$ is a formal $\ZZ$-linear combination of vertices of $G$.
\end{definition}

Divisors are alternatively known as ``chip-configurations'', because we may view a divisor as consisting of stacks of chips on the vertices of $G$.  We write $D(v)$ for the coefficient of the vertex $v$ in the divisor $D$, or alternatively the number of chips of $D$ at $v$.  We define an equivalence relation on the set of divisors on $G$ using chip-firing moves, which we define here.

\begin{definition}
Let $D$ be a divisor on a graph $G$ and fix $v\in V(G)$.  The divisor $D'$ obtained from $D$ by \emph{firing} the vertex $v$ is defined by
\[
D'(w)=\left.
\begin{cases}
D(v) - \val(v) & \text{if } w = v \\
D(w) + \varepsilon(w,v) & \text{if } w \neq v
\end{cases}
\right.
\]
where $\varepsilon(w,v)$ is the number of edges between $w$ and $v$.
\end{definition}

If the vertices of a graph $G$ are labeled $v_1 , \ldots , v_n$, we define the \emph{Laplacian} of $G$, denoted $L(G)$, to be the $n \times n$ matrix with entries
\[
L(G)_{ij} = \left.
\begin{cases}
-\val (v_{i}) & \text{if } i = j \\
\varepsilon(v_{i},v_{j}) & \text{if } i\neq j.
\end{cases}
\right.
\]

Two divisors $D$ and $D'$ are said to be \emph{linearly equivalent} if $D'$ can be obtained from $D$ by a sequence of chip-firing moves.  Equivalently, $D$ and $D'$ are equivalent if their difference is contained in the image of the graph Laplacian $L(G)\ZZ^{V(G)}$.  This defines an equivalence relation on the set of divisors on $G$.

We now recall some definitions related to positivity of divisors on graphs.

\begin{definition}
A divisor $D$ is called \emph{effective} if $D(v) \geq 0$ for all $v\in V(G)$.
\end{definition}

\begin{definition}
The \emph{complete linear series} of a divisor $D$ is
\[
\vert D \vert = \{D'\sim D \vert D' \text{ is effective}\}.
\]
The \emph{support} of an effective divisor $D$ is
\[
\supp(D) = \{v\in V(G) \vert D(v) > 0\}.
\]
\end{definition}

Perhaps the most important property of a divisor on a graph is its Baker-Norine rank.

\begin{definition}
A divisor $D$ is said to have \emph{rank at least} $r$ if, for every effective divisor $E$ of degree $r$, $D-E$ is equivalent to an effective divisor. The \emph{rank} of $D$ is the largest integer $r$ such that $D$ has rank at least $r$.
\end{definition}

Our main interest in this paper is a graph invariant known as the gonality.

\begin{definition}
The \emph{gonality} of a graph $G$ is the smallest degree of a divisor on $G$ with positive rank.
\end{definition}

While it is NP-hard to compute the gonality of a graph \cite{NPhard}, the complexity arises from the sheer number of divisors that one would be required to check.  If we are given a divisor, there is a relatively simple algorithm that determines whether it has positive rank.  To follow the algorithm, one first needs the notion of a reduced divisor.

\begin{definition}
Let $v\in V(G)$. A divisor $D$ is said to be $v$-\emph{reduced} if:
\begin{enumerate}
\item  the divisor $D$ is ``effective away from $v$'' -- that is, for all $w\in V(G) \smallsetminus \{ v \}$, $D(w)\geq 0$, and
\item  for any $U\subset V(G)\smallsetminus \{ v \}$, the divisor obtained by firing all vertices in $U$ is not effective away from $v$.
\end{enumerate}
\end{definition}

Each divisor has a unique $v$-reduced representative.  Moreover, if a divisor $D$ is equivalent to an effective divisor, then its $v$-reduced representative is effective.  With these facts in hand, it becomes relatively easy to check that a divisor $D$ has positive rank.  The divisor $D$ has positive rank if and only if the $v$-reduced divisor equivalent to $D-v$ is effective for all vertices $v$.  Of course, it is crucial in this process that we be able to compute $v$-reduced divisors.  To find $v$-reduced divisors quickly, we have the following algorithm, known as \emph{Dhar's Burning Algorithm}.

\begin{enumerate}
\item  Start a fire at the vertex $v$, placing it in the set of burnt vertices.
\item  For each burnt vertex, place each edge adjacent to it in the set of burnt edges.
\item  If the number of burnt edges adjacent to a vertex $w$ exceeds the number of chips on $w$, place $w$ in the set of burnt vertices.  Then repeat step 2. Otherwise, if no such vertex exists, proceed to step 4.
\item  If the entire set of vertices burn, then $D$ is $v$-reduced.  Otherwise, the divisor obtained by firing all unburnt vertices is effective away from $v$.  Replace $D$ with this equivalent divisor and repeat the process.
\end{enumerate}

\subsection{Expander Graphs}
In this paper we relate the gonality to the expansion properties of a graph.  One of the most well-known measures of graph expansion is known variously as the Cheeger constant, isoperimetric number, or edge expansion constant.  Throughout, if $G$ is a graph and $U\subseteq V(G)$, we define $\partial U$ to be the set of edges with exactly one endpoint in $U$.

\begin{definition} \cite{Cheeger}
The \emph{Cheeger constant} of a graph $G$ is defined to be
\[
h(G):= \min_{0< \vert U \vert \leq \frac{1}{2} \vert V(G) \vert}\frac{\vert \partial U \vert}{\vert U \vert } .
\]
More generally, for any $u \in (0,\frac{1}{2}]$, The $u$-\emph{Cheeger constant} of $G$ is defined to be
\[
h_u(G) := \min_{0< \vert U \vert \leq u \vert V(G) \vert }\frac{ \vert \partial U \vert}{\vert U \vert} .
\]
For values of $u$ greater than $\frac{1}{2}$, we vacuously define $h_u (G)$ to be infinity.
\end{definition}

A second measure of graph expansion is the algebraic connectivity, which we define here.

\begin{definition}
The \emph{algebraic connectivity} $\lambda_{2}$ of a graph $G$ is defined to be the second smallest\footnote{We note with mild exasperation that there is no apparent consistency in the combinatorics literature concerning the indices of the eigenvalues of the graph Laplacian.  They are equally likely to be ordered smallest to largest as largest to smallest, or indexed starting with zero rather than one.  Such issues nevertheless pale in comparison to the wealth of literature that refers simply to the ``second eigenvalue'' without indicating whether this means second smallest or second largest.} eigenvalue of the graph Laplacian $L(G)$.
\end{definition}

For $k$-regular graphs, the algebraic connectivity is related to the Cheeger constant by the well-known Cheeger inequalities:
\[
\frac{1}{2} \lambda_2 \leq h(G) \leq \sqrt{2k\lambda_2} .
\]
One advantage of the algebraic connectivity is that the Cheeger constant is NP-hard to compute, whereas the algebraic connectivity of a graph with $n$ vertices can be computed to any degree of accuracy in $O(n^3)$ time.

The algebraic connectivity can be used to bound the size of vertex separators, which we define below.

\begin{definition} \cite{Eig}
Let $G$ be a graph. A set $C\subset V(G)$ is said to $\emph{separate}$ vertex sets $A,B \subset V(G)$ if
\begin{enumerate}
\item $A$, $B$ and $C$ partition $V(G)$ and
\item no vertex of $A$ is adjacent to a vertex of $B$.
\end{enumerate}
\end{definition}

\begin{lemma} \cite[Theorem~2.8]{Eig}
\label{Lemma:Eig}
Let $G$ be a graph and denote by $d$ the maximal valence of a vertex of $G$.  If $C$ separates vertex sets $A$ and $B$, then
\[
\vert C \vert \geq \frac{4\lambda_2 \vert A \vert \vert B \vert}{d \vert V(G) \vert - \lambda_2 \vert A \cup B \vert}.
\]
\end{lemma}

An important class of expander graphs is that of random regular graphs.  We describe a model for such graphs, which can be found in \cite{Bollobas}.  We fix $n$ sets $W_j$, each consisting of $k$ elements, with $kn$ even.  A \emph{configuration} is a partition of $W = \cup_{j=1}^n W_j$ into 2-element sets.  We let $\Phi_{k,n}$ be the set of all configurations, and turn $\Phi_{k,n}$ into a probability space by taking every configuration to have equal probability of being selected.  Each configuration in $\Phi_{k,n}$ has an associated $k$-regular, $n$-vertex graph, with an edge between two vertices $v_i$ and $v_j$ if and only if there is a set in the partition consisting of an element of $W_i$ and an element of $W_j$.

We say that a property of a random $k$-regular graph holds \emph{asymptotically almost surely} if the probability that the property holds approaches one as $n$ approaches infinity.

\section{Gonality and Edge Expansion}

Our main observation is the following.  If, for every representative of a divisor $D$ on a graph $G$, there exists a ``large'' connected component of $V(G) \smallsetminus \supp (D)$, then $D$ cannot have positive rank.  This is made precise by Theorem \ref{Thm:LowerBound}.

%\begin{theorem}
%\label{Thm:CheegerBound}
%Let $B_{u}(G)$ be the smallest degree of an effective divisor $D$ such that every connected component of $V(G) \smallsetminus \supp(D)$ has size at most $u \vert V(G) \vert $. Then
%\[
%gon(G)\geq \min \left\{ B_{u}(G), h(G)u \vert V(G) \vert \right\}.
%\]
%\end{theorem}

\begin{proof}[Proof of Theorem~\ref{Thm:LowerBound}]
Let $D$ be a divisor of positive rank and assume that $\deg(D) < B_{u}(G)$.  For any representative $D' \in \vert D \vert$, we let $U_{D'}$ denote the union of all connected components of $V(G) \smallsetminus \supp (D')$ of size greater than $u \vert V(G) \vert$.  Note that, by the definition of $B_u (G)$, for any representative $D'$, the set $U_{D'}$ is nonempty.

Choose $D_0 \in \vert D \vert$ such that $U_{D_0}$ is minimal.  That is, for any $D' \in \vert D \vert$, $U_{D'}$ is not strictly contained in $U_{D_0}$.  Now, pick a vertex $v \in U_{D_0}$.  By \cite[Lemma~1.3]{treewidth}, there exist firing sets $A_1 \subseteq A_2 \subseteq \cdots \subseteq A_n \subseteq V(G) \{ v \}$ such that the divisor obtained from $D_{i-1}$ by firing $A_i$ is an effective divisor $D_i$, and $D_n$ is $v$-reduced.  Since by assumption $D$ has positive rank, $D_n(v) > 0$.

Since firing these sets $A_i$ eventually puts at least one chip in $U_{D_0}$, there must exist an index $i$ such that $U_{D_i} \neq U_{D_0}$.  Let $i$ be the smallest such index.  Since $U_{D_0}$ was chosen to be minimal, $U_{D_i} \not\subset U_{D_0} = U_{D_{i-1}}$.  It follows that there exists some connected component $U \subset U_{D_i}$ that intersects the support of $D_{i-1}$.  Since $U$ does not intersect the support of $D_i$, we conclude that $U \subset A_i$.  It follows that $\vert A_i \vert \geq \vert U \vert \geq u \vert V(G) \vert$.

Let $U' \subset U_{D_0}$ be the connected component that contains $v$, and note that by assumption we have $U' \subset V(G) \smallsetminus A_i$. Since $\vert A_i \vert \leq \vert V(G) \smallsetminus U' \vert$, we have $\vert V(G) \smallsetminus A_i \vert \geq \vert U' \vert \geq u \vert V(G) \vert$.  Taking the smaller of $A_i$ and $V(G) \smallsetminus A_i$, we therefore see that $\vert \partial A_i \vert \geq h(G) u \vert V(G) \vert$.  Finally, since $D_i$ is effective, we must have $D_{i-1} (w) \geq \outdeg_w A_i$ for every vertex $w \in A_i$.  Thus, $\deg(D) \geq \vert \partial A_i \vert$, and so $\deg (D) \geq h(G) u \vert V(G) \vert$.
\end{proof}

The following lemma provides a bound on the invariant $B_u (G)$ introduced in Theorem \ref{Thm:CheegerBound} for regular graphs.

\begin{lemma}
\label{Lem:CheegerRegular}
Let $G$ be a $k$-regular graph, and let $D$ be an effective divisor such that
\[
\deg(D) < \frac{h_{u}(G)}{k+h_{u}(G)}\vert V(G) \vert .
\]
Then there exists a connected component of $V(G)\smallsetminus \supp(D)$ with size greater than $u \vert V(G) \vert$.
\end{lemma}

\begin{proof}
Assume that every connected component of $V(G) \smallsetminus \supp(D)$ has size less than or equal to $u \vert V(G) \vert$.  Let $U$ be a connected component of $V(G) \smallsetminus \supp(D)$.  Note that the number of edges with at least one end in $U$, here denoted $e(U)$, is
\[
e(U) = \frac{1}{2} (k \vert U \vert + \vert \partial U \vert ) \geq \frac{k+h_{u}(G)}{2} \vert U \vert.
\]
To see this, note that the number of half-edges with one end in $U$ is exactly $k \vert U \vert$.  Since the internal edges of $U$ contribute twice to this sum and the edges that leave $U$ contribute only once, we arrive at the equality above. The inequality is given by the definition of the $u$-Cheeger constant.

Let $e(V(G)\smallsetminus \supp(D))$ denote the number of edges in the complement of the support of $D$. We then have
\[
e(V(G)\smallsetminus \supp(D)) = \sum_U e(U),
\]
where the sum is over connected components of $V(G) \smallsetminus \supp (D)$.  By the above, therefore, we have
\[
e(V(G)\smallsetminus \supp(D)) \geq \sum_U \frac{k+h_{u}(G)}{2} \vert U \vert
\]
\[
 = \frac{k+h_{u}(G)}{2} \vert V(G) \smallsetminus \supp(D) \vert
\]
\[
 > \frac{k+h_{u}(G)}{2}\left(1 - \frac{h_{u}(G)}{k+h_{u}(G)}\right)\vert V(G) \vert = \frac{k}{2} \vert V(G) \vert.
\]

Since the total number of edges in $G$ is $\frac{k}{2} \vert V(G) \vert$, this is impossible.  Therefore, there exists a connected component of the complement of $\supp (D)$ of size greater than $u \vert V(G) \vert$.
\end{proof}

Together, Theorem \ref{Thm:LowerBound} and Lemma \ref{Lem:CheegerRegular} yield a lower bound on the gonality of regular graphs in terms of the Cheeger constant.

%\begin{corollary}
%\label{Cor:LowerBound}
%Let $G$ be a $k$-regular graph. Then, for any $u$, we have
%\[
%\gon(G) \geq \min \left\{ \frac{h_{u}(G)}{k+h_{u}(G)} \vert V(G) \vert , h(G)u \vert V(G) \vert \right\} .
%\]
%\end{corollary}

\begin{proof}[Proof of Theorem~\ref{Thm:CheegerBound}]
By Lemma \ref{Lem:CheegerRegular}, for any effective divisor $D$ with $\deg(D) < \frac{h_{u}(G)}{k+h_{u}(G)} \vert V(G) \vert$, there exists a connected component of $V(G) \smallsetminus \supp (D)$ of size greater than $u \vert V(G) \vert$.  Since all effective divisors of this degree have this property, we see that
\[
\frac{h_{u}(G)}{k+h_{u}(G)} \vert V(G) \vert \leq B_u (G) .
\]
The result then follows from Theorem \ref{Thm:LowerBound}.
\end{proof}

%\begin{theorem}
%\label{Thm:RandomThreeRegular}
%Let $G$ be a random 3-regular graph.  Then
%\[
%\gon(G) \geq 0.071 \vert V(G) \vert
%\]
%asymptotically almost surely.
%\end{theorem}

\begin{proof}[Proof of Theorem~\ref{Thm:RandomThreeRegular}]
By \cite{KM93}, the Cheeger constant of a random 3-regular graph is asymptotically almost surely bounded below by $\frac{1}{4.95}$.  Similarly, \cite[Theorem~3]{Kolesmald} provides lower bounds above which the $u$-Cheeger constant of a random regular graph is bounded asymptotically almost surely.  For 3-regular graphs, this bound is optimized for $u \approx 0.36$, at which point it is $0.24$.  By Theorem \ref{Thm:CheegerBound}, therefore, the gonality of a random 3-regular graph is asymptotically almost surely greater than or equal to
\[
\min \left\{ \frac{0.24}{3.24} \vert V(G) \vert, \frac{0.36}{4.95} \vert V(G) \vert \right\} = 0.072 \vert V(G) \vert .
\]
\end{proof}

While our main interest is in 3-regular graphs, we note the following amusing fact.

\begin{proposition}
Let $\gon (k)$ be the supremum over all $\ell$ such that the gonality of a random $k$-regular graph on $n$ vertices is asymptotically almost surely at least $\ell n$.  Then
\[
\lim_{k \to \infty} \gon (k) \geq \frac{1}{3} .
\]
\end{proposition}

\begin{proof}
Let $i(k)$ be the supremum over all $\ell$ such that the Cheeger constant of a random $k$-regular graph is asymptotically almost surely at least $\ell k$.  By \cite{Bollobas}, $\lim_{k \to \infty} i(k) \geq \frac{1}{2}$.  It follows that
\[
\lim_{k \to \infty} \gon (k) \geq \lim_{k \to \infty} \frac{i(k)k}{k+i(k)k} = \frac{\frac{1}{2}}{1+\frac{1}{2}} = \frac{1}{3} .
\]
\end{proof}

\section{Gonality and Spectral Expansion}

Determining the Cheeger constant of a graph is an NP-complete problem \cite{Kaibel04}.  Using the algebraic connectivity, however, we can provide a lower bound on graph gonality that can be computed in polynomial time.  This bound also does not require the graph to be regular, and can therefore be applied in a more general setting than Theorem \ref{Thm:CheegerBound}.  To obtain this bound, we first prove a proposition about divisors that act as separators.

\begin{proposition}
\label{Prop:Separator}
Let $D$ be a divisor on a graph $G$ such that all connected components of $V(G) \smallsetminus \supp(D)$ have size less than $\frac{1}{2}\vert V(G) \vert$. Then there exist sets $A, B \subset V(G)$ such that $\supp(D)$ separates $A$ and $B$ and
\[
\frac{1}{3} \vert V(G) \smallsetminus \supp(D) \vert \leq \vert A \vert \leq \frac{2}{3} \vert V(G) \smallsetminus \supp(D) \vert .
\]
\end{proposition}

\begin{proof}
Let $A$ be the maximal union of connected components of $V(G) \smallsetminus \supp(D)$ such that $\vert A \vert \leq \frac{2}{3} \vert V(G) \smallsetminus \supp(D) \vert$, and let $B$ be the union of the remaining connected components.  If $\frac{1}{3} \vert V(G) \smallsetminus \supp(D) \vert \leq \vert A \vert $, then we are done.  We may therefore suppose that $\vert A \vert < \frac{1}{3} \vert V(G) \smallsetminus \supp(D) \vert$.  In this case, we move a component $U$ of $B$ into $A$.

Note that, by our assumption that $A$ is maximal, we must have $\vert A \cup U \vert > \frac{2}{3} \vert V(G) \smallsetminus \supp (D) \vert$.  Then, since each connected component of $V(G) \smallsetminus \supp(D)$ has size less than half the vertices of $G$, we see that $\frac{1}{3} \vert V(G) \smallsetminus \supp(D) \vert \leq \vert U \vert \leq \frac{1}{2} \vert V(G) \vert$.  In that case, we may let $A$ = $U$, and let $B$ be the union of the remaining connected components.
\end{proof}

We now use Proposition \ref{Prop:Separator} in conjunction with Lemma \ref{Lemma:Eig} to find a lower bound for the gonality entirely in terms of the algebraic connectivity, the maximum valence of a vertex in $G$, and $\vert V(G) \vert$.

%\begin{theorem}
%\label{Thm:SpectralBound}
%Let $G$ be a graph, and let $d$ be the maximum valence of a vertex in $G$. Then
%\[
%\gon(G) \geq \frac{\vert V(G) \vert}{2\lambda_2} \left[ -(7\lambda_2 + 9d) + 3\sqrt{9\lambda_2^2 + 14d\lambda_2 + 9d^2} \right] .
%\]
%\end{theorem}

\begin{proof}[Proof of Theorem~\ref{Thm:SpectralBound}]
Let $D$ be a divisor on $G$ of positive rank.  By Theorem \ref{Thm:CheegerBound}, there exists a representative for $D$ such that every connected component of $V(G) \smallsetminus \supp (D)$ has size less than $\frac{1}{2} \vert V(G) \vert$.  Therefore, by Proposition \ref{Prop:Separator} we have that $\supp(D)$ separates vertex sets $A$ and $B$ such that
\[
\frac{1}{3} \vert V(G) \smallsetminus \supp(D) \vert \leq \vert A \vert \leq \frac{2}{3} \vert V(G) \smallsetminus \supp(D) \vert .
\]
Let $x = \frac{\vert V(G) \vert}{\vert A \vert}$.  By assumption, $\frac{1}{3} \leq x \leq \frac{2}{3}$, and
\[
\vert A \vert \vert B \vert = x(1-x) \vert V(G) \smallsetminus \supp (D) \vert ^2 .
\]
On the interval $[\frac{1}{3}, \frac{2}{3}]$, the function $x(1-x)$ has a single critical point at $x = \frac{1}{2}$, but since this is a downward opening parabola we know this critical point is a maximum.  Therefore, $x(1-x)$ attains its minimum value at the endpoints $x = \frac{1}{3}, \frac{2}{3}$.  Thus we have that
\[
x(1-x) \vert V(G) \smallsetminus \supp (D) \vert^2 \geq \frac{2}{9}(\vert V(G) \vert - \vert \supp (D) \vert )^2 .
\]
Employing Lemma \ref{Lemma:Eig} we have
\[
|\supp (D)| \geq \frac{4\lambda_2 \vert A \vert \vert B \vert}{d \vert V(G) \vert -\lambda_2 \vert A\cup B \vert} \geq \frac{\frac{8 \lambda_2}{9}( \vert V(G) \vert - \vert \supp (D) \vert )^2}{d \vert V(G) \vert -\lambda_2( \vert V(G) \vert - \vert \supp(D) \vert )}.
\]
Solving for $\vert \supp (D) \vert$, we obtain:
\[
\lambda_2 \vert \supp (D) \vert^2 + (7\lambda_2 + 9d)\vert V(G) \vert \vert \supp (D) \vert - 8\lambda_2 \vert V(G) \vert^2 \geq 0 .
\]
We may now apply the quadratic formula.
\[
\vert \supp(D) \vert \geq \frac{1}{2\lambda_2} \left[ -(7\lambda_2 + 9d)\vert V(G) \vert + \sqrt{(7\lambda_2 + 9d)^2 \vert V(G) \vert^2 + 32\lambda_2^2 \vert V(G) \vert^2} \right] .
\]
\[
= \frac{\vert V(G) \vert}{2\lambda_2} \left[ -(7\lambda_2 + 9d) + 3\sqrt{9\lambda_2^2 + 14d\lambda_2 + 9d^2} \right] .
\]
\end{proof}

Note that the Laplacian of a graph can be constructed easily given a set of vertices and edge relations, and using the QR algorithm, eigenvalues of the Laplacian can be approximated to any degree of accuracy in $O(n^{3})$ time. This means that given an arbitrary graph, we can use Theorem \ref{Thm:SpectralBound} to obtain a lower bound on the gonality in a relatively small amount of time.

Although it is not as good as our previous bound, we record here what Theorem \ref{Thm:SpectralBound} can tell us about random 3-regular graphs.

\begin{corollary}
Let $G$ be a random 3-regular graph. Then, for any $\epsilon > 0$, we have
\[
\gon(G) \geq (0.0486 - \epsilon ) \vert V(G) \vert
\]
asymptotically almost surely.
\end{corollary}

\begin{proof}
By \cite{Friedman08}, for any $\epsilon > 0$, the algebraic connectivity of a random $k$-regular graph is asymptotically almost surely bounded below by $k - 2\sqrt{k-1} - \epsilon$.  Taking $k = 3$, we evaluate the right hand side of the inequality in Theorem \ref{Thm:SpectralBound} to obtain:
\[
\gon(G) \geq \frac{\vert V(G) \vert}{2(3-2\sqrt{2})} \left[-(7(3-2\sqrt{2}) + 27) + 3\sqrt{9(3 - 2\sqrt{2})^{2} + 42(3-2\sqrt{2}) + 81} \right]
\]
\[
= 0.0486 \vert V(G) \vert .
\]
\end{proof}

\section{An Example}

As an example, we compute these bounds for the Pappus graph pictured in Figure \ref{Fig:Pappus}.  The Pappus graph is notable for having the largest Cheeger constant of any 3-regular graph with at most 20 vertices.  As such, it is a good candidate for testing our results.

\begin{figure}
\begin{tikzpicture}[
vertex_style/.style={circle,ball color=black},
edge_style/.style={ultra thick, black}]

\useasboundingbox (-5.05,-5.3) rectangle (5.1,5.25);

\begin{scope}[rotate=90]
   \foreach \x/\y in {0/1,60/2,120/3,180/4,240/5,300/6}{
      \node[vertex_style] (\y) at (canvas polar cs: radius=1.25cm,angle=\x){};
   }
   \foreach \x/\y in {0/7,60/8,120/9,180/10,240/11,300/12}{
      \node[vertex_style] (\y) at (canvas polar cs: radius=2.5cm,angle=\x){};
   }

   \foreach \x/\y in {0/13,60/14,120/15,180/16,240/17,300/18}{
      \node[vertex_style] (\y) at (canvas polar cs: radius=.625cm,angle=\x){};
   }

\end{scope}

\foreach \x/\y in {13/16,14/17,15/18}{
   \draw[edge_style] (\x) -- (\y);
}

\foreach \w/\x/\y/\z in {1/7/14/18,2/8/15/13,3/9/16/14,4/10/17/15,5/11/18/16,6/12/13/17}{
   \draw[edge_style] (\w) -- (\x);
   \draw[edge_style] (\w) -- (\y);
   \draw[edge_style] (\w) -- (\z);
}

\foreach \x/\y in {7/8,8/9,9/10,10/11,11/12,12/7}{
   \draw[edge_style] (\x) -- (\y);
}
\end{tikzpicture}
\caption{The Pappus graph}
\label{Fig:Pappus}
\end{figure}
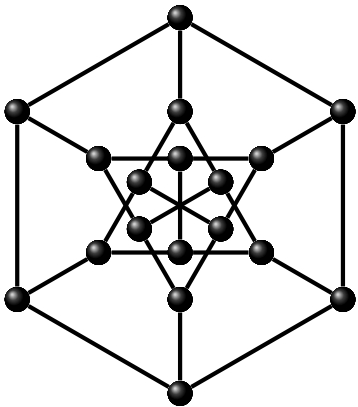

As a first estimate of the gonality of the Pappus graph, we apply Theorem \ref{Thm:LowerBound}.  The $u$-Cheeger constants for the Pappus graph are provided in Table \ref{Table}.  Evaluating the expression
\[
\min \left\{ \frac{h_{u}(G)}{3+h_{u}(G)} \vert V(G) \vert ,h(G)u \vert V(G) \vert \right\}
\]
for all $u$ in Table \ref{Table}, we find that the bound given by Theorem \ref{Thm:LowerBound} is optimized for $u = \frac{1}{3}$, giving a bound of $\frac{1}{4}\vert V(G) \vert$.  Thus $\gon (G) \geq 4.5$.

\begin{table}
\label{Table}
  \centering
    \begin{tabular}{|c|c|}
    \hline
    $u$ & $h_{u}(G)$\\
    \hline \hline
    $1/18$ & $3$\\
    \hline
    $2/18$ & $2$\\
    \hline
    $3/18$ & $5/3$\\
    \hline
    $4/18$ & $3/2$\\
    \hline
    $5/18$ & $7/5$\\
    \hline
    $6/18$ & $1$\\
    \hline
    $7/18$ & $1$\\
    \hline
    $8/18$ & $1$\\
    \hline
    $9/18$ & $7/9$ \\
    \hline
    \end{tabular}
  \caption{u-Cheeger constants of the Pappus graph}\label{table}
\end{table}

A better estimate is provided by Theorem \ref{Thm:SpectralBound}.  The second smallest eigenvalue of the Pappus graph is $\lambda_2 = 3-\sqrt{3}$.  We may now evaluate the bound given by Theorem \ref{Thm:SpectralBound}.
\[
\gon(G) \geq \frac{18}{2(3-\sqrt{3})} \left[ -7(3-\sqrt{3} + 27)+3\sqrt{9(3-\sqrt{3})^{2} + 42(3-\sqrt{3}) + 81} \right] = 5.04.
\]
Since the gonality is an integer, we now know that the gonality of the Pappus graph is at least 6.  It is expected that this is the highest possible gonality of a graph of genus 10.  We now show that the gonality of the Pappus graph is exactly 6.

\begin{proposition}
Let $G$ be the Pappus graph. Then $\gon(G) = 6$.
\end{proposition}

\begin{proof}
Let $D$ be the divisor on the Pappus graph illustrated in Figure \ref{Fig:Divisor}.  Note that if we fire every vertex in the complement of the outer ring once, these chips will migrate to the outer ring.  Therefore it suffices to show that we can get at least one chip to any vertex in the innermost ring.  By symmetry it suffices to check that we can get a chip to a single vertex in the innermost ring.  To do this, we will employ Dhar's burning algorithm on $D$, with $v$ being a vertex in the innermost ring.

Now, the fire burns the edge connecting $v$ to an adjacent vertex in the inner ring.  Each of these vertices is connected to two distinct vertices in the middle ring by one edge each. Since the support of our divisor is the middle ring, the fire must stop having only burned $v$ and the adjacent vertex, so we fire all other vertices once. This chip firing move puts a chip on $v$, which shows that $D$ has rank at least 1.
\end{proof}

\begin{figure}
\begin{tikzpicture}[
vertex_style/.style={circle,ball color=black},
edge_style/.style={ultra thick, black},
chip_style/.style ={circle, ball color = white}]
\useasboundingbox (-5.05,-5.3) rectangle (5.1,5.25);

\begin{scope}[rotate=90]
   \foreach \x/\y in {0/1,60/2,120/3,180/4,240/5,300/6}{
      \node[chip_style] (\y) at (canvas polar cs: radius=1.25cm,angle=\x){};
   }
   \foreach \x/\y in {0/7,60/8,120/9,180/10,240/11,300/12}{
      \node[vertex_style] (\y) at (canvas polar cs: radius=2.5cm,angle=\x){};
   }

   \foreach \x/\y in {0/13,60/14,120/15,180/16,240/17,300/18}{
      \node[vertex_style] (\y) at (canvas polar cs: radius=.625cm,angle=\x){};
   }

\end{scope}

\foreach \x/\y in {13/16,14/17,15/18}{
   \draw[edge_style] (\x) -- (\y);
}

\foreach \w/\x/\y/\z in {1/7/14/18,2/8/15/13,3/9/16/14,4/10/17/15,5/11/18/16,6/12/13/17}{
   \draw[edge_style] (\w) -- (\x);
   \draw[edge_style] (\w) -- (\y);
   \draw[edge_style] (\w) -- (\z);
}

\foreach \x/\y in {7/8,8/9,9/10,10/11,11/12,12/7}{
   \draw[edge_style] (\x) -- (\y);
}
\end{tikzpicture}
\caption{A divisor on the Pappus graph, highlighted in white}
\label{Fig:Divisor}
\end{figure}
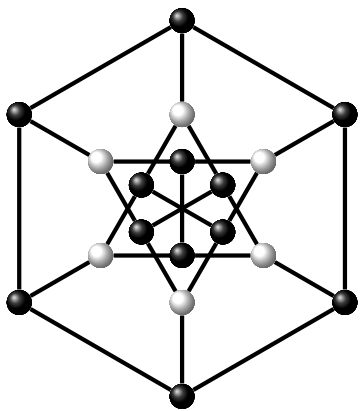

\bibliographystyle{alpha}
\bibliography{Paperbib}

\end{document}